\documentclass[12pt]{amsart}
\usepackage{amsmath}

\newtheorem{numbered}{}[section]
\newtheorem{thm}[numbered]{Theorem}

\newtheorem{remark}[numbered]{Remark}

\newtheorem{prop}[numbered]{Proposition}
\newtheorem{cor}[numbered]{Corollary}

 \numberwithin{equation}{section}

\begin{document}

\title{Eisenstein series and modular differential equations}
\author{Abdellah Sebbar \& Ahmed Sebbar}

\address{Department of Mathematics and Statistics, University of Ottawa, Ottawa Ontario K1N 6N5 Canada}
\address{Institut de Math\'ematiques de Bordeaux,
Universit\'e Bordeaux 1 351, cours de la Lib\'eration F-33405
Talence cedex}

\email{asebbar@uottawa.ca}
\email{ahmed.sebbar@math.u-bordeaux1.fr}



\begin{abstract}The purpose of this paper is to solve various differential
equations having Eisenstein series as coefficients using various tools and techniques. The solutions
are given in terms of modular forms, modular functions and
equivariant forms. 
\end{abstract}
\maketitle

\section{Introduction}
In this paper we study differential equations of the form
$$y''\,+\,\alpha E_4(z)y\,=\,0,$$
where $E_4(z)$ is the weight 4 Eisenstein series, for different
values of $\alpha$ yielding solutions of different kinds. These
differential equations are interesting in many respects. The
graded differential ring of modular forms and their derivatives is
generated by the Eisenstein series $E_2$, $E_4$ and $E_6$, and the
differential being $\frac{d}{2\pi i dz}$. Differentiating twice a
homogeneous element of this ring would increase the weight by 4
and multiplying by $E_4$ would have the same effect. This
justifies the form of the above second order linear differential
equation with a particular interest towards the case where
$\alpha$ is a rational multiple of $\pi^2$ (resulting from a
double differentiation).

 One particular second order linear differential equation
that has drawn the attention of Klein ~\cite{klein}, Hurwitz
~\cite{hur}, Van der Pol ~\cite{vdp} takes the form
\begin{equation}\label{vdp}
y''\,+\,\frac{\pi^2}{36}E_4(z)y\,=\,0.
\end{equation}
 Van der Pol,
~\cite{vdp}, noticed that the Riccati equation attached to
\eqref{vdp} is
\begin{equation}\label{vdp1}
\frac{6}{i\pi}u'\,+\,u^2\,=\,E_4,
\end{equation}
and thanks to Ramanujan's identities, ~\cite{ram}, one has
$u=-E_2$ as a solution to \eqref{vdp1}. Thus one can recover the
solution to \eqref{vdp} in terms of the classical discriminant
function.

The purpose of this paper is to investigate  differential
equations similar to \eqref{vdp1}. Namely, we study the Riccati
equation, or the corresponding linear ODE, of the form
\begin{equation}\label{seb}
\frac{k}{i\pi}u'\,+\,u^2\,=\,E_4\,  ,\qquad k=1,\cdots,6.
\end{equation} and the corresponding second order ODE
\begin{equation}\label{vdp2}
y''\,+\,\frac{\pi^2}{k^2}E_4(z)y\,=\,0.
\end{equation}
It turns out that in all but one case, modular functions arise as
solutions. However, when $k=1$, we will see that the solution has
very special features leading to an example of the so-called
equivariant forms developed by the authors in a forthcoming paper.
In four cases, these equations can be transformed to Schwarzian
differential equations which can be solved using techniques
 developed in a previous paper by one of the
author~\cite{s1}.

It should be mentioned that differential equations with modular
forms as coefficients appear in various contexts including in
physics. Indeed, particular differential equations similar to
\eqref{vdp2} but with the presence of a first order term, namely
$$y''-4i\pi E_2(z)y'+\frac{44}{5}\pi^2E_4(z)y\,=\,0\ ,$$
appear in ~\cite{mil} as a consequence of the character formula
for the Virasoro  model and the theory of vertex operator
algebras. This leads to one of the Ramanujan identities.
Surprisingly, the technique developed in ~\cite{m-s,s1} in the
level 5 case a few years before ~\cite{mil} can provide the same
solution with purely number-theoretic arguments. Another similar
equation
$$y''-\frac{1}{3}i\pi E_2(z)y'+\frac{2}{3}\pi^2E_4(z)y\,=\,0$$
appears in ~\cite{sam} to classify rational conformal theories. On
the other hand, an extensive study of some modular differential
equations involving $E_2$ and $E'_2$ as coefficients can be found
in ~\cite{k-k} of the form
\begin{equation}\label{kk1}
y''-\frac{k+1}{6}E_2(z)y'+\frac{k(k+1)}{12}E'_2(z)y\,=\,0.
\end{equation}
 These equations are studied from the hypergeometric
equations point of view. However, the only common feature with our
modular differential equations is the fact that they are both
homogeneous in the differential ring of quasi-modular forms.
Moreover, the equation \eqref{kk1} can be symmetrized to have the
form
\begin{equation}\label{kk2}
\partial_{k+2}\partial_ky-\frac{k(k+1)}{144}E_4(z)y\,=\,0
\end{equation}
where $\partial_ky=y'-\frac{k}{12}E_2(z)y$ and using
$12E'_2(z)=E_2^2-E_4$. However, the equations \eqref{vdp2} and
\eqref{kk2} are different in their nature and in their solutions.

The paper is organized as follows: In Section 2, we recall the
correspondence between second order linear ODE, the Riccati
equations, and the Schwarz differential equation. In Section 3 we
deal with the classical case of Van der Pol corresponding to the
case $k=6$ in \eqref{seb}. Section 4 deals with the cases
$k=2,3,4,5$ in one shot using analytic properties of
the Schwarzian derivative derived from a previous work ~\cite
{m-s}, ~\cite{s1}. The most interesting case to us is when $k=1$.
In addition of providing the solution to the corresponding ODE, we
find that this solution satisfies an intriguing equivariance
property.

{\em Acknowledgmens.} We thank the referee for his comments and
for pointing out the reference \cite{MK}.

\section{Linear and non-linear ODE's}
In this section, we recall some classical connections between
three kinds of differential equations.

Let $R(z)$ be a meromorphic function on a domain $D$ of the
complex plane. The first  equation under consideration is the
second order linear differential equation of the form
\begin{equation}\label{lode}
y''\,+\,\frac{1}{4}R(z)y\,=\,0\,,
\end{equation}
for which the space of local solutions is a two-dimensional vector
space. The second differential equation is non-linear and has the
form
\begin{equation}\label{sch} \{f,z\}\,=\,R(z)\,,
\end{equation}
 involving the Schwarz derivative
 $$
 \{f,z\}\,=\,2\left(\frac{f''}{f'}\right)'\,-\,\left(\frac{f''}{f'}\right)^2.
 $$
 If $y_1$ and $y_2$ are two linearly independent
 solutions to \eqref{lode}, then $f=y_1/y_2$ is a solution to
 \eqref{sch}. Moreover, a function $g$ is a solution to
 \eqref{sch} if and only if $g$ is a non-constant linear
 fractional transform of $f$. Conversly, if $f$ is a solution to
 \eqref{sch} which is locally univalent, then $y_1=f/\sqrt{f'}$ and
 $y_2=1/\sqrt{f'}$ are two linearly independent solutions to
 \eqref{lode}.

 Meanwhile, if we set $u=f''/f'$, we see that \eqref{sch} yields
 a Riccati equation
 \begin{equation}\label{ric0}
 2u'\,-\,u^2\,=\,R(z).
 \end{equation}
 Finally, one can go directly from solutions to \eqref{lode} to
 solution to \eqref{ric0} by taking  logarithmic derivatives.

 In the remaining of this paper, we will provide solutions to
 certain differential equations by favoring the form that leads to
 a simple closed solution.

\section{Eisenstein series as coefficients of ODE's}

The study of differential equations having modular forms as
solutions or as coefficients has a long history going back to
Dedekind, Schwarz, Klein, Poincar\'e, Hurwiz, Ramanujan, Van der
Pol, Rankin and several others. The differential equations under
consideration are Schr\"odinger type equations  (or
Sturm-Liouville equations) $y''+Q(z)y=0$ where $Q(z)$ is the
potential. When $Q(z)$ satisfies some invariance properties, the
Schr\"odinger equation leads to a variety of interesting theories.
For instance, when the potential $Q(z)$ is periodic, then we have
Mathieu's equation or Hill's equation widely studied in the last
two centuries. When the potential is doubly periodic, we have
Lam\'e equation which motivated extensive research leading to the
modern theory of integrable systems. Also as it was mentioned in
the introduction, differential equations with modular forms as
coefficients have been extensively studied in physics and in
relation with hypergeometric functions and quasimodular forms.

In this paper we consider equations with a potential that is
automorphic or modular explicitly given by the Eisenstein series.
We will propose solutions to these equations and show that their
existence is very natural.
 Let us recall some definitions.

The Dedekind eta-function is defined on the upper half-plane
$\mathcal H=\{z\in{\mathbb C}\ ,\ \Im z>0\}$ by
$$\eta(z)\,=\,q^{\frac{1}{24}}\,\prod_{n\geq 1}(1-q^n)\,  ,\quad
q=e^{2\pi i z}\, ,\, z\in{\mathcal H}.
$$
The discriminant $\Delta$ is defined by
$$
\Delta(z)\,=\,\eta(z)^{24}.
$$
The Eisenstein series $E_2(z)$, $E_4(z)$ and $E_6(z)$ are defined
for $z\in{\mathcal H}$ by
$$
E_2(z)\,=1-24\,\sum_{n\geq1}\,\frac{nq^n}{1-q^n}\,=\,1-24\,\sum_{n\geq
1}\,\sigma_1(n)\,q^n,
$$
$$
E_4(z)\,=1+240\,\sum_{n\geq1}\,\frac{n^3q^n}{1-q^n}\,=\,1+240\,\sum_{n\geq
1}\,\sigma_3(n)\,q^n,
$$
and
$$
E_6(z)\,=1-504\,\sum_{n\geq1}\,\frac{n^5q^n}{1-q^n}\,=\,1-504\,\sum_{n\geq
1}\,\sigma_5(n)\,q^n,
$$
 where $\sigma_k(n)$ is the sum of the $k$-th powers  of all
the positive divisors of $n$. The function $E_4$ and $E_6$ are
modular forms for $\mbox{SL}_2(\mathbb Z)$ of respective weight 4
and 6. In ~\cite{vdp}, Van der Pol studied the equation
$$
y''\,+\,\frac{\pi^2}{36}E_4(z)y\,=\,0,
$$
and was later generalized to higher order differential equations
in ~\cite{vdp1}. This equation appears also in ~\cite{hur} and in
~\cite{klein}.

If we set
$$u\,=\,\frac{6}{i\pi}\,\frac{y'}{y},$$
then the differential equation satisfied by $y$ transforms into a
Riccati equation
$$
\frac{6}{i\pi}\,u'\,+\,u^2\,=\,E_4(z).
$$
In the meantime, we have differential relations between Eisenstein
series also known as the Ramanujan identities ~\cite{ram}
\begin{equation}\label{ram1}
\frac{1}{2\pi i}\frac{d}{dz}\,E_2(z)\,=\,\frac{1}{12}(E_2^2-E_4),
\end{equation}
\begin{equation}\label{ram2}
\frac{1}{2\pi i}\frac{d}{dz}\,E_4(z)\,=\,\frac{1}{3}(E_2E_4-E_6),
\end{equation}
\begin{equation}\label{ram3}
\frac{1}{2\pi
i}\frac{d}{dz}\,E_6(z)\,=\,\frac{1}{2}(E_2E_6-E_4^2).
\end{equation}
Using \eqref{ram1} and the fact that
\begin{equation}\label{ram0}
E_2(z)\,=\,\frac{1}{2\pi i}\,\frac{\Delta'}{\Delta}\, ,
\end{equation}
we have
\begin{prop}
The Riccati equation
$$
\frac{6}{i\pi}\,u'\,+\,u^2\,=\,E_4(z)
$$ has $u=-E_2(z)$ as a solution, and the corresponding linear
differential equation
$$
y''\,+\,\frac{\pi^2}{36}E_4(z)y\,=\,0
$$ has as a solution
$$y\,=\,{\Delta}^{-\frac{1}{12}}\,=\,{\eta}^{-2}.$$
\end{prop}
For the remainder of this paper, we study similar equations
obtained by varying the coefficient of $u'$ in the Riccati
equation. Namely, we look at the following equations
\begin{equation}\label{ricgen}
\frac{k}{i\pi}\,u'\,+\,u^2\,=\,E_4(z)\ ,\quad k=1,\cdots,6\,.
\end{equation}
While the case $k=6$ has been the subject of this section, the
cases $k=2,\cdots,5$ will be treated in the next section, and the
case $k=1$ will be the subject of the following section.

\section{The modular case}
We follow a different approach in dealing with the cases $2\leq
k\leq 5$ in \eqref{ricgen} by using the equivalent form involving
the Schwarzian differential equations
\begin{equation}\label{schgen}
\{f,z\}\,=\,\frac{k^2}{4\pi^2}\,E_4(z)\ ,\quad 2\leq k\leq 5.
\end{equation}
Suppose that $f$ is a modular function for a finite index subgroup
$\Gamma$ of $\mbox{SL}_2(\mathbb Z)$, that is a meromorphic
function defined on $\mathcal H$ such that $f(\gamma\cdot
z)\,=\,f(z)$ for $\gamma\in\Gamma$, $z\in{\mathcal H}$ and
$\gamma\cdot z=\frac{az+b}{cz+d}$ for $\gamma=\binom{a\ b}{c\ d}$,
and $f$ extending to a meromorphic function on ${\mathcal
H}\cup\{\infty\}$. The derivative $f'$ is a weight 2 meromorphic
modular form, and one can show that if $g$ is a modular form of
weight $k$, then $kgg''-(k+1)g'^2$ is a modular form of weight
$2k+4$. Applying this to $f'$, one obtains that $2f'f'''-3f''^2$
is a modular form of weight 8, and if divided by $f'^2$, it yields
$\{f,z\}$ which is then a modular form of weight 4. A more
interesting case occurs when the group $\Gamma$ is of genus 0 in
the sense that the compactification of the Riemann surface
$\Gamma\backslash{\mathcal H}$ is of genus 0. Indeed, if $f$ is a
Hautpmodul for $\Gamma$, that is a generator of the function field
of the compact Riemann surface, then $\{f,z\}$ is actually a
weight 4 modular form for exactly the normalizer of $\Gamma$ in
$\mbox{SL}_2(\mathbb R)$, ~\cite{m-s}. Furthermore, the form
$\{f,z\}$ has double poles at the elliptic fixed points of
$\Gamma$ (which are the critical points of $f$) and is holomorphic
everywhere including at the cusps.

Recall that the space of holomorphic modular forms of weight 4 for
$\mbox{SL}_2(\mathbb Z)$ is one-dimensional and is generated by
$E_4$. Thus, to be able to use the above analysis to get $\{f,z\}$
being a scalar multiple of $E_4$, one should look for $\Gamma$
being normal in $\mbox{SL}_2(\mathbb Z)$, having genus 0, and
having no elliptic points. This is the case if one takes $\Gamma$
to be one of the principal congruence groups $\Gamma(n)$, $2\leq
n\leq 5$. By choosing a Hauptmodul $f$ for each such $\Gamma(n)$,
one is able to identify exactly which multiple of $E_4$ is the
modular form $\{f,z\}$. It turns out that this depends on the
level $n$, which is also the cusp width at $\infty$. Indeed, for
$2\leq n\leq 5$, a Hauptmodul $f_n$ for $\Gamma(n)$ satisfies
$f_n(z+n)=f_n(z)$, $z\in{\mathcal H}$, hence $f_n$ has a Fourier
expansion in $q=\exp(2\pi iz/n)$, and if one chooses $f_n$ to
vanish at $\infty$, then we can take $f_n(z)=q+\mbox{O}(q)$ as
 $q\longrightarrow\infty$. In this case, one can show
directly from the expression of the Schwarz derivative that
$$\{f_n,z\}\,=\,\frac{4\pi^2}{n^2}\,+\,\mbox{O}(q) \ \mbox{as}\ q\longrightarrow\infty.$$
Notice that any other Hauptmodul will lead to the same conclusion
since it would be a linear fractional transform of $f_n$ and thus
having the same Schwarz derivative. In conclusion, we have
\begin{thm}
Let $f_n$ be any Hauptmodul for $\Gamma(n)$, $2\leq n\leq 5$, then
$f_n$ is a solution to the equation
$$\{f,z\}\,=\,\frac{4\pi^2}{n^2}\,E_4(z).$$
Furthermore, if we set
$$u_n=-\frac{n}{2\pi i}\frac{f_n''}{f_n'}\ ,\quad 2\leq n\leq5\, ,$$
then $u_n$ is a solution of the Riccati equation
$$
\frac{n}{i\pi}u'\,+\,u^2\,=\,E_4.
$$
\end{thm}
To make this theorem  effective, we shall provide explicit
expressions for the above-mentioned Hauptmoduln. These expressions
are classical and can be found for instance in ~\cite{m-s} and
~\cite{s1}.

\begin{center}
\begin{tabular}{|c|c|}
\hline Group & Hauptmodul\\
\hline
&\\
$\Gamma(2)$&$\left(\frac{\eta(2z)}{\eta(z/2)}\right)^8$ \\
&\\
\hline
&\\
$\Gamma(3)$&
$\left(\frac{\eta(3z)}{\eta(z/3)}\right)^3$\\
&\\
\hline
&\\
$\Gamma(4)$&$\frac{\eta(z/2)\eta(4z)^2}{\eta(z/4)^2\eta(2z)}$\\
&\\
\hline
&\\
$\Gamma(5)$&$q^{1/5}\prod_{n\geq
1}(1-q^n)^{\left(\frac{n}{5}\right)}$ \\
&\\
\hline
\end{tabular}

$\left(\frac{\ \cdot\ }{\ \cdot\ }\right)$ is the Legendre symbol
\end{center}

This completes the study of the equations \eqref{ricgen} for the
cases $2\leq k\leq 5$, and the solutions are expressed in terms of
modular functions.

\begin{remark}{\rm
The case $k=6$ settled in Proposition 3.1 cannot use the method of
this section since $\Gamma(6)$ is not of genus 0. The same applies
to  the case $k=1$ since $\Gamma(1)=\mbox{SL}_2(\mathbb Z)$ has
elliptic fixed points. This case shall be studied in the next
section.}
\end{remark}

\section{The equivariant case}
We now focus on the case $k=1$  in ~\eqref{ricgen} or its
equivalent form in terms of the second order linear differential
equation. We have
\begin{thm}
The function $y=E_4'\,\Delta^{-1/2}$ is a solution to
\begin{equation}\label{k1}
y''\,+\,\pi^2E_4\,y\,=\,0.
\end{equation}
\end{thm}
\begin{proof}
Using \eqref{ram0} one has $\Delta'=2\pi i\Delta E_2$, and  using
\eqref{ram1}, one finds that
$$\Delta''\,=\,-\frac{13\pi^2}{3}\Delta E_2^2+\frac{\pi^2}{3}\Delta
E_4.$$ It follows that for $y=E_4'\,\Delta^{-1/2}$ we have
$$y''\,+\,\pi^2E_4\,y\,=\,\Delta^{-\frac{1}{2}}\,[E_4'''-2\pi
iE_4''E_2-\frac{5\pi^2}{6}E_4'E_2^2+\frac{5\pi^2}{6}E_4'E_4].$$
Meanwhile, since the differential ring of modular forms and their
derivatives is $\left({\mathbb C}[E_2,E_4,E_6],\frac{1}{2\pi
i}\frac{d}{dz}\right)$, one is able to express every term
algebraically by means of the generators $E_2$, $E_4$ and $E_6$.
Indeed, using Ramanujan identities \eqref{ram1}, \eqref{ram2} and
\eqref{ram3}, one has
$$E_4'\,=\,\frac{2\pi i}{3}\,(E_2E_4-E_6),$$
$$E_4''\,=\,-\frac{5\pi^2}{9}\left(E_2^2E_4-2E_2E_6+E_4^2\right),$$and
$$E_4'''\,=\,-\frac{5\pi^3
i}{9}\left(E_4E_2^3+3E_4^2E_2-3E_2^2E_6-E_4E_6\right).$$ Now,
substituting into the expression of $y''\,+\,\pi^2E_4\,y$ shows it
is identically equal to 0.

\end{proof}
We mentioned in the previous section that if $g$ is a modular form
of weight $k$, then $(k+1)g'^2-kgg''$ is a modular form of weight
$2k+4$. If $g=E_4$, then $5E_4'^2-4E_4E_4''$ is a holomorphic
modular form of weight 12, which clearly vanishes at $\infty$ and
thus it is a multiple of $\Delta$. Investigating the leading
coefficient yields
$$5E_4'^2-4E_4E_4''\,=\,-3840\pi^2\,\Delta.$$
While $y=(1/i\pi\sqrt{3840})E_4'\,\Delta^{-1/2}$ is still a
solution of \eqref{k1}, we then would like to find a function $f$
such that $1/\sqrt{f'}=y$. This, according to Section 2, will
provide us with a solution to the corresponding Schwarzian
equation $\{f,z\}=4\pi^2\,E_4(z)$. The function would satisfy
\begin{align*}
f'(z)\,&=\,\frac{1}{y^2}\\
 &=\,\frac{-3840\pi^2\,\Delta}{E_4'^2}\\
 &=\,\frac{5E_4'^2-4E_4E_4''}{E_4'^2}\\
 &=\,1+4\left(\frac{E_4}{E_4'}\right)'.
 \end{align*}
 Thus, we have the following

\begin{thm}
The function defined on $\mathcal H$ by
\begin{equation}\label{equiv}
f(z)\,=\,z\,+\,4\,\frac{E_4}{E_4'}
\end{equation}
is a solution to
$$
\{f,z\}\,=\,4\pi^2\,E_4(z)
$$
\end{thm}
As a consequence, we have
\begin{cor}The function
$$
u\,=\,\frac{1}{i\pi}\frac{E_4''}{E_4'}-E_2
$$ is a solution to the Riccati equation
$$\frac{1}{i\pi}u'+u^2=E_4$$
\end{cor}

Unlike the previous section, the solution \eqref{equiv} is
 not a modular function and in fact it is not invariant under any non-identity element of
 $\mbox{SL}_2(\mathbb Z)$. More precisely it has the following
  surprising
equivariance properties.
\begin{thm}
The function $f$ defined on $\mathcal H$ by \eqref{equiv}
satisfies
\begin{equation}\label{equiv1}
f\left(\frac{az+b}{cz+d}\right)\,=\,\frac{af(z)+b}{cf(z)+d}\
,\quad z\in{\mathcal H}\ ,\quad \binom{a\ b}{c\
d}\in\mbox{SL}_2(\mathbb Z).
\end{equation}
\end{thm}
\begin{proof}
To show that $f(\gamma\cdot z)=\gamma\cdot f(z)$ holds for all
$\gamma\in \mbox{SL}_2(\mathbb Z)$, it suffices to prove that it
holds  for the two generators  $\binom{1\ 1}{0\ 1}$ and $\binom{0\
-1}{1\ \  0}$ of $\mbox{SL}_2(\mathbb Z)$. It is clear that
$f(z+1)=f(z)+1$ and we only need to check that $f(-1/z)=-1/f(z)$.
Since $E_4(-1/z)=z^4E_4(z)$ we have
$E_4'(-1/z)=4z^5E_4(z)+z^6E_4'(z)$. Hence
\begin{align*}
f\left(\frac{-1}{z}\right)\,&=\,\frac{-1}{z}
\,+\,\frac{4E_4(z)}{4zE_4(z)+z^2E_4'(z)}\\
&=\,\frac{-E_4'}{4E_4+zE_4'}\\
&=\,\frac{-1}{z+4\frac{E_4}{E_4'}}\\
&=\,\frac{-1}{f(z)}.
\end{align*}
This concludes the proof.
\end{proof}

The function $f$ above is an example of the so-called  {\em
equivariant forms} developed by the authors in a forthcoming paper
~\cite{2s} regarding meromorphic functions defined on the upper
half-plane and commuting with the action of a discrete group, see
also \cite{br,sm}.  In fact, for a subgroup of the modular group,
a subclass
 of these equivariant forms consists of functions $f$ defined by
$$f(z)=z+k\frac{g(z)}{g'(z)}\ ,$$
where $g(z)$ is a modular form of weight $k$, \cite{sm}. What
makes the equivariance property work is that the function
$$F(z)\,=\,\frac{k}{f(z)-z}\,=\,\frac{g'(z)}{g(z)}
$$
transforms as
$$(cz+d)^{-2}\,F\left(\frac{az+b}{cz+d}\right)\,=\,F(z)\,+\,\frac{kc}{cz+d},$$
which, in the terms of M. Knopp, \cite{MK}, makes $F(z)$  a
modular integral with period function $\frac{kc}{cz+d}$, or a
quasi-modular form in the language of D. Zagier. Further
connections between these subjects and equivariant forms will be
made explicit in \cite{2s}.
\begin{prop}
The derivative of $f$ is given by
$$f'(z)\,=\,-3840\pi^2\frac{\Delta}{{E'_4}^2}.$$
\end{prop}
\begin{proof} We have
$$f'(z)\,=\,\frac{5{E'_4}^2-4E_4E''_4}{{E'_4}^2}.$$
It is a known fact that if $f$ is a modular form of weight $k$,
then $(k+1)f'^2-kff''$ is a modular form of weight 2k+4. Hence
$5{E'_4}^2-4E_4E''_4$ is a modular form of weight 12 vanishing at
$\infty$. It turns out that $5{E'_4}^2-4E_4E''_4=-3840\Delta$, and
the proposition follows.
\end{proof}

 \end{document}